\documentclass[11pt]{amsart}
\usepackage{amscd}
\usepackage{amsfonts}
\usepackage{amsmath,amsthm,hyperref}
\usepackage[all]{xy}
\usepackage{amsmath,amsthm,hyperref}
\usepackage{amsmath,amssymb,amsthm,latexsym}
\usepackage{amscd}
\usepackage{color}

\newtheorem{theorem}{Theorem}[section]
\newtheorem{proposition}[theorem]{Proposition}

\newtheorem{corollary}[theorem]{Corollary}
\newtheorem{lemma}[theorem]{Lemma}

\newtheorem{conjecture}[theorem]{\bf Conjecture}
\numberwithin{equation}{section}
\theoremstyle{remark}
\newtheorem{remark}[theorem]{Remark}

\newcommand{\CC}{\mathcal{C}}
\newcommand{\J}{\mathcal{J}}
\newcommand{\NM}{\perp\!\!M}
\newcommand{\Ind}{\operatorname{Ind}}
\newcommand{\Span}{\operatorname{Span}}
\newcommand{\trace}{\operatorname{trace}}

\newcommand{\blue}{\textcolor{blue}}

\def\dd{\mathrm{d}}
\def\<{\langle}
\def\>{\rangle}

\begin{document}
\title[On the index of minimal $2$-tori in the $4$-sphere]{On the index of minimal $2$-tori in the $4$-sphere}

\author[Rob Kusner $\&$ Peng Wang]{Rob Kusner $\&$ Peng Wang}

\address {Department of Mathematics, University of Massachusetts, Amherst, MA 01003, USA}
\email{profkusner@gmail.com, kusner@umass.edu}
\address{School of Mathematics and Statistics, FJKLMAA, Key Laboratory of Analytical Mathematics and Applications (Ministry of Education), Fujian Normal University, Fuzhou 350117, P. R. China}
\email{pengwang@fjnu.edu.cn, netwangpeng@163.com}

\date{\today}
\maketitle

\begin{center} Dedicated to Professor Eugenio Calabi for his century of inspiration\end{center}

\begin{abstract}
In this note we prove that any minimal $2$-torus in $S^4$ has Morse index at least $6$, with equality if and only if it is congruent to the Clifford torus in some great $S^3\subset S^4$.
For a minimal $2$-torus in $S^n$ with vanishing Hopf differential, we show that its index is at least $n+3$, and that this estimate is sharp:
the equilateral $2$-torus fully embedded in
$S^5\subset S^n$ as a homogeneous minimal surface in $S^n$ has index exactly $n+3$.
\end{abstract}



\section{Introduction}

The study of minimal subvarieties in the unit $n$-sphere $S^n$ is an important topic in global differential geometry. The work  of Simons \cite{Simons} has been especially influential: in particular, he investigated the second variation formula, observing that any closed minimal subvariety in $S^n$ is unstable and estimating its Morse index.   The great $2$-sphere and the Clifford torus are the simplest minimal surfaces in $S^3$, with Morse indices $1$ and $5$, respectively. Urbano \cite{Urbano} showed these are the only minimal surfaces in $S^3$ with index at most $5$.  This is a key ingredient in the recent proof \cite{Marques} by Marques and Neves of the Willmore conjecture in $S^3$.  As they point out \cite{Marques2}, strong evidence for the analogous conjecture in higher codimension would be a generalization of Urbano's result to minimal $2$-tori in $S^n$, which is the principal purpose of this note.

Urbano's proof relies heavily on the codimension-$1$ assumption to find global eigenfunctions for the Jacobi operator associated to the second variation of area.  In higher codimension we need to consider sections of the normal bundle instead of functions, and finding useful global eigensections is a challenging task.
For this reason, we first focus our attention on minimal $2$-tori in $S^4$, where the normal-vector-valued holomorphic quadratic differential $\CC$
--- introduced by Calabi \cite{Calabi} and whose square is the scalar-valued Hopf differential $\mathcal{H}$ ---
provides two natural test sections of the normal bundle.  Furthermore, the codimension-$2$ situation lets us show the normal bundle is trivial provided $\mathcal{H}$ does not vanish.  In this case we can also define an almost complex structure $\J$ on the normal bundle. The action of $\J$ on $\CC$ then provides more test sections, leading to our main result:

\begin{theorem}\label{thm-main} Let $f:T^2\rightarrow S^4$ be a minimal torus in the unit $4$-sphere. Then we have  $\Ind(f)\geq 6$ and equality holds if and only if $f$ is congruent to the Clifford torus in some great $S^3\subset S^4$.
\end{theorem}
\noindent
Theorem \ref{thm-main} is a corollary of Theorem  \ref{prop-H-no} and Proposition \ref{prop-H-va-s4}.

\subsection{Plan of the paper}  We begin by reviewing (in Section 2) the basic theory of minimal surfaces in $S^n$. Then we prove (in Section 3) our main results, considering first the case of minimal tori with nonvanishing Hopf differential $\mathcal{H}$, and next the {\em superminimal} case with vanishing $\mathcal{H}$.  For the latter case, we give an index estimate for {\em all} minimal $2$-tori in $S^n$.  We conclude (in Section 4) with some remarks about the higher genus case as well as the nonorientable case.

\subsection{Acknowledgements} We thank Misha Karpukhin, Ping Li, Xiang Ma, Chao Xia, and the anonymous referee for their comments on earlier drafts of this paper.
Rob Kusner was supported in part by National Science Foundation grants PHY-1607611 at the Aspen Center for Physics in June 2017 and DMS-1439786 while visiting ICERM during Spring 2018.  Peng Wang was supported in part by   Projects 11971107 and 12371052 of NSFC and CSC.

\section{Minimal surfaces in $S^n$}

Consider an isometric  immersion $f:M\rightarrow S^n$ of a closed Riemannian surface $M$ into the unit sphere $S^n\subset\mathbb R^{n+1}$.
Its mean curvature vector field ${H}:= \trace(\hbox{II}) \in \Gamma(\NM)$ is a section of the normal bundle $\NM$, where $\hbox{II}$ denotes the $\NM$-valued second fundamental form dual to the shape operator $\mathcal A$ of $f$.  Define the associated section $\tilde{\mathcal{A}}$ of the symmetric endomorphism bundle of $\NM$ so that for any $V, W\in\Gamma(\NM)$
\[\<\tilde{\mathcal A}(V),W\>:=\<\mathcal A(V),\mathcal A(W)\>,\] and observe that the trace of $\tilde{\mathcal{A}}$ equals the square length
$S:=|\mathcal{A}|^2=|\hbox{II}|^2$ of the second fundamental form (see \cite[pages 67-69]{Simons} for more details).

\subsection{Variation formulas for the Area functional and the Morse index}  Suppose $f_t$ is a variation of $f_0=f:M\rightarrow S^n$ with $\frac{\partial}{\partial t}(f_t)|_{t=0}=V\in \Gamma(\NM)$, and let $A(t)$ denote the area of $f_t$. Then it is well known \cite{Simons, CW}
that the {\em first variation formula} is
\[ A'(0)=-\int_M\< {H}, V\>.\]
So for a minimal $f:M\rightarrow S^n$ the mean curvature {is}  ${H}\equiv0$, and in this case the {\em second variation formula} is
\begin{equation}\label{index-form}
A''(0)=-\int_M\<\Delta^{\perp}V+\mathcal{R}(V)+\tilde{\mathcal{A}}(V),V\>=-\int_M\< \mathcal{L}(V), V\>=:Q(V,V).
\end{equation}
Here the Laplace operator $\Delta^{\perp}$ and the {\em Jacobi operator}
\begin{equation}\label{eq-L} \mathcal{L}:=\Delta^{\perp}+\mathcal{R}+\tilde{\mathcal{A}}=\Delta^{\perp}+2+\tilde{\mathcal{A}}
\end{equation}
act on sections of the normal bundle $\NM$, noting that for surfaces in $S^n$, the partial Ricci endomorphism simplifies to $\mathcal{R}(V)=2V$.
In the special case of an oriented surface $f:M\rightarrow S^3$, we have $V=gN$ with $N$ the orienting unit normal vector field and $g$ a function on $M$,
and so the Jacobi operator further simplifies to
\[\mathcal L(V)=(\Delta g +2g+Sg)N,\] where $\Delta$ is the Laplacian for functions on $M$.
For any codimension, the Jacobi operator $\mathcal L$ is a strongly elliptic, self-adjoint operator acting on sections of $\NM$, with finite-dimensional eigenspaces and a discrete spectrum
of real eigenvalues $\{\alpha_j\}$ which is bounded from below.

The {\em Morse index} $\Ind(f)$ of the minimal immersion $f:M\rightarrow S^n$ is defined to be the index of the quadratic form $Q$ associated (\ref{index-form}) with the  {second variation}  of $f$: this is the maximal dimension of a subspace of normal sections on which $Q$ is negative definite, that is, the sum of the dimensions of the
eigenspaces of $\mathcal{L}$ belonging to its negative eigenvalues.

\subsection{Computations in a complex coordinate {and}  Calabi-Hopf differentials} Now assume $M$ is orientable, and regard it as a Riemann surface
with the conformal structure induced from the minimal immersion $f:M\rightarrow S^n$.
Since $f$ is conformal, in any local complex coordinate chart $z$ on $M$ one has {$\<f_z,f_z\>\equiv0$}
and the induced metric has conformal factor $|f_z|^2=\frac{1}{2}e^{2\rho}$.
The minimal surface equation  {$f_{z\bar{z}}+|f_z|^2f\equiv0$}
becomes the first of the {\em structure equations}
\begin{equation}\label{eq-moving}
  \left\{\begin{split}
  f_{\bar{z}z}&=-\frac{1}{2}e^{2\rho}f,\\
  f_{zz}&=\Omega+2\rho_zf_z,\\
  \psi_{z}&=\nabla^{\perp}_z\psi -2e^{-2\rho}\<\psi,\Omega\>f_{\bar z}\\
\end{split}\right.
\end{equation}
for the adapted local frame field $\{f_{\bar{z}}, f_z, \psi\}$.
Here $\{f_{\bar{z}}, f_z\}$ is a local tangent frame field, $\psi$ is an arbitrary section of the normal bundle $\NM$ for $f:M\rightarrow S^n$, and
the connection $\nabla^{\perp}_z$ takes values in the complexified normal bundle $\NM\otimes\underline{\mathbb C}$.
The terms on the right-hand sides of the second and third equations correspond to the splitting into normal and tangential components:
in particular, the normal part of $f_{zz}$
defines the local section $\Omega$ of $\NM\otimes\underline{\mathbb C}$.
One must instead tensor $\NM$ with the square of the canonical complex line bundle $\mathbb K$ for $M$ to get a global section:
the normal-vector-valued holomorphic quadratic {\em Calabi differential} $$\CC:=\Omega\dd z^2 \in \Gamma(\NM\otimes{\mathbb K}^2),$$ whose square is the scalar-valued holomorphic quartic {\em Hopf differential}
\begin{equation*}\label{eq-Hopf}
\mathcal{H}:=\<\CC,\CC\>=\<\Omega,\Omega\>\dd z^4= \<f_{zz},f_{zz}\>\dd z^4 \in \Gamma(\NM\otimes{\mathbb K}^4).
\end{equation*}

Now set
$e_1=e^{-\rho}(f_z+f_{\bar{z}}),~~e_2=ie^{-\rho}(f_z-f_{\bar{z}})$, so that $\{e_1,e_2\}$ is an orthonormal frame for $TM$, with $\{\omega_1, \omega_2\}$ its dual coframe.
Let $\{\psi_{\nu}\}$ be an orthonormal frame for $\NM$. Then
the second fundamental form $ {\hbox{II}}=:\sum_{\nu,i,j}h^{\nu}_{ij}\psi_{\nu}\omega_i\omega_j$ and we have
\begin{equation}\label{eq-Omega-h}
\Omega=\frac{1}{4}e^{ 2\rho}\sum_{\nu}\left(h^{\nu}_{11}-h^{\nu}_{22}-2ih^{\nu}_{12}\right)\psi_{\nu}.\end{equation}
Since $H=0$, we get
\[|\Omega|^2=\frac{1}{4}e^{ 4\rho}\sum_{\nu}\left((h^{\nu}_{11})^2+(h^{\nu}_{12})^2\right).\]
Therefore we can use $\Omega$ to express the  {square length} of the second fundamental form
\[S=| {\hbox{II}}|^2=8e^{-4\rho}|\Omega|^2\]
for $f$, as well as express the Gauss, Codazzi and Ricci equations
\begin{equation}\label{eq-integrability}
  \left\{\begin{split}
-K+1&=\frac{S}{2}=4e^{-4\rho}|\Omega|^2,\\
  \nabla^{\perp}_{\bar{z}}\Omega&=0,\\
R_{z\bar{z}}\psi&=\nabla^{\perp}_{\bar z}\nabla^{\perp}_z\psi-\nabla^{\perp}_z \nabla^{\perp}_{\bar z}\psi=2e^{-2\rho}(\<\psi,\Omega\>\bar\Omega-\<\psi,\bar\Omega\>\Omega),\\
\end{split}\right.
\end{equation}
and an elementary computation verifies:
\begin{lemma}
 The Laplacian $\Delta^{\perp}$ on smooth normal sections $\Gamma(\NM)$ has the form
 \begin{equation}\label{eq-laplacian-def}
   \Delta^{\perp}\psi=2e^{-2\rho}\left(\nabla^{\perp}_{\bar z}\nabla^{\perp}_z\psi+\nabla^{\perp}_z \nabla^{\perp}_{\bar z}\psi\right).
 \end{equation}
\end{lemma}
\begin{proof}
By definition (see, for example  {\cite[page 41]{CW}}) we have
\[\Delta^{\perp}\psi=\nabla^{\perp}_{e_1}\nabla^{\perp}_{e_1}\psi+\nabla^{\perp}_{e_2}\nabla^{\perp}_{e_2}\psi-\nabla^{\perp}_{\nabla_{e_1}e_1}\psi-\nabla^{\perp}_{\nabla_{e_2}e_2}\psi.\]
Choosing $e_1=e^{-\rho}(f_z+f_{\bar{z}})$ and $e_2=ie^{-\rho}(f_z-f_{\bar{z}})$ as above, we have
\[
\begin{split}\nabla^{\perp}_{\bar z}\nabla^{\perp}_z\psi=&\frac{1}{4}e^{2\rho}\left(\nabla^{\perp}_{e_1}\nabla^{\perp}_{e_1}\psi+\nabla^{\perp}_{e_2}\nabla^{\perp}_{e_2}\psi
-i\nabla^{\perp}_{e_1}\nabla^{\perp}_{e_2}\psi+i\nabla^{\perp}_{e_2}\nabla^{\perp}_{e_1}\psi\right.\\
&\left.+e_1(\rho)\nabla^{\perp}_{(e_1-ie_2)}\psi+ie_2(\rho)\nabla^{\perp}_{(e_1-ie_2)}\psi\right), \quad \text{and also}
\end{split}\]
\[
\begin{split}\nabla^{\perp}_{z}\nabla^{\perp}_{\bar z}\psi=&\frac{1}{4}e^{2\rho}
\left(\nabla^{\perp}_{e_1}\nabla^{\perp}_{e_1}\psi+\nabla^{\perp}_{e_2}\nabla^{\perp}_{e_2}\psi
+i\nabla^{\perp}_{e_1}\nabla^{\perp}_{e_2}\psi-i\nabla^{\perp}_{e_2}\nabla^{\perp}_{e_1}\psi\right.\\
&\left.+e_1(\rho)\nabla^{\perp}_{(e_1+ie_2)}\psi-ie_2(\rho)\nabla^{\perp}_{(e_1+ie_2)}\psi\right).
\end{split}\]
Summing, we obtain
\begin{equation}\label{eq-com}\nabla^{\perp}_{\bar z}\nabla^{\perp}_z\psi+\nabla^{\perp}_z\nabla^{\perp}_{\bar z}\psi=\frac{1}{2}e^{2\rho}\left(\nabla^{\perp}_{e_1}\nabla^{\perp}_{e_1}\psi
+\nabla^{\perp}_{e_2}\nabla^{\perp}_{e_2}\psi+e_1(\rho)\nabla^{\perp}_{e_1}\psi+e_2(\rho)\nabla^{\perp}_{e_2}\psi\right).\end{equation}
On the other hand, by \eqref{eq-moving} we have
\[\nabla_{f_z}f_z=2\rho_zf_z,\ \nabla_{f_z}f_{\bar z}=0.\]
Expansion with respect to $e_1,e_2$ gives
\[\nabla_{e_1}e_1=-e_2(\rho)e_2,\ \nabla_{e_2}e_2=-e_1(\rho)e_1\]
and substituting these into \eqref{eq-com} yields \eqref{eq-laplacian-def}.
\end{proof}
\subsection{A Simons-type identity for the Calabi differential}
Applying the Laplacian \eqref{eq-laplacian-def} to the normal-vector-valued Calabi differential, and using \eqref{eq-integrability}, gives this complex Simons-type \cite{Simons} identity
\begin{equation}\label{eq-Omega}
\Delta^{\perp}\Omega=4e^{-4\rho}(\<\Omega,\Omega\>\bar\Omega-\<\Omega,\bar\Omega\>\Omega).
\end{equation}

\subsection{Further remarks concerning the Jacobi eigenfields of Simons}

The following lemma generalizes the estimate of $\Ind(f)$ in \cite{Simons} (for the case of minimal hypersurfaces, see \cite{Urbano} and \cite{Perdomo}).
\begin{lemma}\label{lemma-simons} Let $f:M\rightarrow S^n$ be a minimal immersion of a closed Riemann surface $M$ with local complex coordinate $z$.
Let $Z\in\mathbb R^{n+1}\setminus\{0\}$ be a constant vector and let $Z^{\perp}$ be the normal bundle projection of $Z$.
Then $Z^{\perp}$ is an eigenfield for the Jacobi operator of $f$ with eigenvalue $-2$.
Moreover, the dimension of $\Span_{\mathbb R}$$\left\{Z^{\perp}|Z\in\mathbb R^{n+1}\right\}$ is $(n+1)$ if $f$ is not totally geodesic.
\end{lemma}

\begin{proof}
By  {\cite[Lemmas 5.1.4 and 5.1.6]{Simons},} we  {only need  to} prove\footnote{Our proof of Lemma \ref{lemma-simons} illustrates the use of a local complex coordinate.  The general result should be well known, and indeed a version appears in recent unpublished notes of Rick Schoen.}
 the last statement of the lemma. We argue indirectly, assuming $f$ is not totally geodesic, and thus $\Omega\not\equiv0$.
Let $\{E_0,\cdots,E_{n}\}$ be the standard basis of $\mathbb R^{n+1}$. Without loss of generality, we assume that $E_0^{\perp}$ is a linear combination of $E_1^{\perp},\cdots, E_{n}^{\perp}$. So there exist constants $c_1,\cdots,c_n$ such that
\[E_0^{\perp}=\sum_{j=1}^{n}c_jE_j^{\perp},~~~\,\,\hbox{ that is, } \,\,~~E_0=\sum_{j=1}^{n}c_jE_j+bf+b_1f_z+\bar{b}_1f_{\bar z},\]
where $b$ and $b_1$ are some functions. Taking derivatives, we have
\[b_zf+b_{1z}f_z+\bar{b}_{1z}f_{\bar z}+bf_z+b_1f_{zz}+\bar{b}_1f_{z\bar z}\equiv0.\]
By \eqref{eq-moving}, the normal bundle projection of the left-hand side is $b_1f_{zz}^{\perp}=b_1\Omega$, and thus $b_1\equiv0$, which implies $bf$ is constant. However $|f|=1$ and $f$ is not constant, so we have $b\equiv0$. Thus $\{E_0,\cdots,E_{n}\}$ is linearly dependent, which is a contradiction.
\end{proof}

\vspace{3mm}
Let $E\in \mathbb R^{n+1}$ be an arbitrary constant vector field. Then we have
\[E^\perp=E-\<E,f\>f-2e^{-2\rho}\<E,f_z\>f_{\bar z}-2e^{-2\rho}\<E,f_{\bar z}\>f_{ z}.\]
 Using \eqref{eq-moving}, we have the identity
\begin{equation}\label{eq-ENz}
  (E^\perp)_{z}= -2e^{-2\rho}\<E,\Omega\>f_{\bar z}-2e^{-2\rho}\<E,f_{\bar z}\>\Omega.
\end{equation}
from which we obtain  {
$\nabla^{\perp}_{z}E^{\perp}=-2e^{-2\rho} \<E,f_{\bar z}\>\Omega$} and
\begin{equation}\label{eq-ez-2}\<\Omega,E^{\perp}\>_{\bar z}=-2e^{-2\rho}|\Omega|^2\<E,f_{z}\>.\end{equation}

As a consequence, we obtain  {the following result.}
\begin{proposition}Let $f$ be a  {non-totally geodesic} minimal surface in $S^n$, that is, with Calabi differential $\CC=\Omega \dd z^2\not\equiv0$. Then the following are equivalent:
\begin{itemize}
\item $f$ is not full,
\item $\<\Omega,E^{\perp}\>\equiv0$ for some $E\in\mathbb R^{n+1}\setminus \{0\}$,
\item $\<E^{\perp},E^{\perp}\>\equiv  {\hbox{constant}}$ for some $E\in\mathbb R^{n+1}\setminus  \{0\}$.
\end{itemize}
\end{proposition}
\begin{proof}
If $f$ is not full, the result is obvious.
Let $E\in\mathbb R^{n+1}\setminus  \{0\}$.  {We have}
\[0\equiv\<E^{\perp},E^{\perp}\>_z=-4e^{-2\rho} \<E,f_{\bar z}\>\<\Omega,E^{\perp}\>,\]
 {if $\<E^{\perp},E^{\perp}\>\equiv \hbox{constant}$,}
so either $\<\Omega,E^{\perp}\>\equiv 0$  or $\<E,f_{\bar z}\>\equiv0$.  If the former, then by \eqref{eq-ez-2} we have
\[0\equiv\<\Omega,E^{\perp}\>_{\bar z}=-2e^{-2\rho}|\Omega|^2\<E,f_{z}\>,\] 
so  $\<E,f_{\bar z}\>\equiv0$.  Hence we have $\<E,f_{\bar z}\>\equiv0$ in both cases.  Then $0\equiv\<E,f_{\bar z}\>_z=-\frac12e^{2\rho}\<E,f\>$ by \eqref{eq-moving}, and $f$ lies in the great subsphere
perpendicular to $E$, that is, $f$ is not full.
\end{proof}

\section{Minimal tori in $S^n$}

In this section we focus on a surface $M$ of genus $1$, so that $\Omega$ defines two {\em global} sections of the normal bundle $\NM$, which we can use to construct test sections.  If the genus were greater than~$1$, then $\Omega$ might not be globally defined on $M$.

Assume that $M=T^2$ is a torus with a global complex coordinate $z$ (on its universal covering space $\mathbb C$).
Thus the Hopf differential $\mathcal{H}=\<\Omega,\Omega\>\dd z^4=a\dd z^4$ for some constant $a\in\mathbb C$.
By a homothety of $\mathbb C$, we can assume either $a=0$ or $a=1$, and we will consider these two cases separately.

\subsection{Technical details on the Calabi-Hopf differentials}
First recall some basic results about  $\Omega=\Omega_1+i\Omega_2$ with $\Omega_1,\Omega_2\in \Gamma(\NM)$. We can assume
 \[\<\Omega_1,\Omega_2\>=0, ~~~~~\<\Omega,\Omega\>=\<\Omega_1,\Omega_1\>-\<\Omega_2,\Omega_2\>=a=0 \hbox{ or } 1.\]
\noindent
Since in either case $\<\Omega_1,\Omega_2\>=0$, by \eqref{eq-Omega} and \eqref{eq-Omega-h} we have
\begin{equation}\label{eq-Omega-1}\Delta^{\perp}\Omega_1=-8e^{-4\rho}|\Omega_2|^2\Omega_1, \,~~~ \hspace{3mm}\Delta^{\perp}\Omega_2=-8e^{-4\rho}|\Omega_1|^2\Omega_2,
\end{equation}
and
\begin{equation}\label{eq-Omega-A}
\tilde{\mathcal{A}}(\Omega_j)=8e^{-4\rho}|\Omega_j|^2\Omega_j,~ j=1,2.\end{equation}
So by  \eqref{eq-L}  we obtain
\begin{equation}\label{eq-Omega-L}
\left\{\begin{split}
\mathcal{L}(\Omega_1)&=2\left(1+4e^{-4\rho}(|\Omega_1|^2-|\Omega_2|^2)\right)\Omega_1,\\
\mathcal{L}(\Omega_2)&=2\left(1-4e^{-4\rho}(|\Omega_1|^2-|\Omega_2|^2)\right)\Omega_2.\\
\end{split}\right.
\end{equation}

\subsection{The case with nonvanishing Hopf differential  $\mathcal{H}$}

\subsubsection{The index estimate}
\begin{theorem}\label{prop-H-no} Let  $f:T^2\rightarrow S^n$  be a minimal torus with $\mathcal{H}\neq0$. Then
\begin{enumerate}\item
$\Ind(f)\geq n+2$.
\item
If $n=4$, then $\Ind(f)\geq 6$ and equality holds if and only if $f$ is congruent to the Clifford torus in some great $S^3\subset S^n$.
 \item
If $f$ is full with $n>4$ and $K\equiv0$, then $\Ind(f)\geq n+3$.
 \end{enumerate}
\end{theorem}
\begin{proof}
Since $\mathcal{H}\neq0$, we can assume that $\<\Omega,\Omega\>=|\Omega_1|^2-|\Omega_2|^2=1$ as discussed above.
From equation \eqref{eq-Omega-L} we have $\mathcal{L}(\Omega_1)=2(1+4e^{-4\rho})\Omega_1$, yielding the inequality
\footnote{Note that for a surface $M$ of genus greater than 1, $\Omega$ is not globally defined on $M$. 
Nevertheless it may be interesting to multiply $\Omega$ by a cut-off function (as in \cite{Doris}) to estimate the first eigenvalue of the Jacobi operator.}
\[-\int_{T^2}\<\Omega_1,\mathcal{L}(\Omega_1)\>\dd A=-2\int_{T^2}(1+4e^{-4\rho})|\Omega_1|^2 \dd A<-2\int_{T^2}|\Omega_1|^2 \dd A.\]
So the first eigenvalue $\alpha_1$ of $\mathcal{L}$ satisfies $\alpha_1<-2$, and $\Ind(f)\geq n+1+1=n+2$, giving  (1).

Now consider item (2): if $f$ lies in some great $S^3$, then (2) follows from Urbano's theorem \cite{Urbano}.
 Thus we  may assume  $f$ is full in $S^4$.  We will show by construction of test eigensections that the eigenspace belonging to $\alpha=-2$ has dimension {\em at least} $6$. Since $\alpha_1<-2$, we have $\Ind(f)\geq7.$

First, since the normal bundle $\NM$ is orientable, and since $|\Omega_1|\geq 1$, there  {exist} some function $\theta$ on $T^2$, and global smooth unit sections $\psi_3$ and $\psi_4$ such that $\Omega_1=\cosh \theta \psi_3$ and
$\Omega_2=\sinh \theta \psi_4$. Since the Calabi differential $\CC=\Omega\dd z^2$ is holomorphic, that is, $\nabla^{\perp}_{\bar z}\Omega=0$, we have
 {$\nabla^{\perp}_{z}\psi_3=i\theta_z\psi_4.$}
We define a smooth almost complex structure $\J$ on the normal bundle $\NM$ of $f$ such that
 {$
\J\psi_3=\psi_4,\ \J\psi_4=-\psi_3.
$}
It is straightforward to show that  {$\nabla^{\perp} \J\equiv0.$}
As a consequence, by \eqref{eq-Omega-1} and \eqref{eq-Omega-A} we obtain
\[\mathcal{L}(\J\Omega_1)= 2\J\Omega_1,\quad \mathcal{L}(\J\Omega_2)= 2\J\Omega_2.\]
We claim that at least one of $\{\J\Omega_1,e_0^\perp,\cdots, e_4^\perp\}$ and $\{\J\Omega_2,e_0^\perp,\cdots, e_4^\perp\}$ is linearly independent.
As a consequence, $\Ind(f)\geq 1+6=7$ holds. We will prove this claim by contradiction.

Suppose  {that}  both $\{\J\Omega_1,e_0^\perp,\cdots, e_4^\perp\}$ and $\{\J\Omega_2,e_0^\perp,\cdots, e_4^\perp\}$ are linearly {\em dependent}.  Since
$\{e_0^\perp,\cdots, e_4^\perp\}$ is linearly independent, $\J\Omega_1+E^\perp=0$ for some $E\in\mathbb R^5$. Then by \eqref{eq-ENz} we have
\[\nabla^{\perp}_{z}(\J\Omega_1)-2e^{-2\rho}\<\J\Omega_1,\Omega\>f_{\bar z}-2e^{-2\rho}\<E,\Omega\>f_{\bar z}-2e^{-2\rho}\<E,f_{\bar z}\>\Omega=0.\]
Since $\nabla^{\perp}_{z}(\J\Omega_1)=-i\theta_z\Omega$, from the above equation we obtain
\[i\theta_{z}=-2e^{-2\rho}\<E,f_{\bar z}\>.\]
Taking one more derivative on both sides, by use of \eqref{eq-moving}, we obtain
\begin{equation}\label{eq-Ef}
 i(\theta_{zz}+2\rho_z\theta_z)=\<E,f\>.
\end{equation}
In particular, $\theta_{zz}+2\rho_z\theta_z$ is pure imaginary.

On the other hand, we also have that $\J\Omega_2+\hat E^\perp=0$ for some $\hat E\in\mathbb R^5$. Then again by \eqref{eq-ENz} we have
\[\nabla^{\perp}_{z}(\J\Omega_2)-2e^{-2\rho}\<\J\Omega_2,\Omega\>f_{\bar z}-2e^{-2\rho}\<\hat E,\Omega\>f_{\bar z}-2e^{-2\rho}\<\hat E,f_{\bar z}\>\Omega=0.\]
Since $\nabla^{\perp}_{z}(\J\Omega_2)=\theta_z\Omega$, from the above equation we obtain
\[\theta_{z}=2e^{-2\rho}\<\hat{E},f_{\bar z}\>,\hbox{ and hence } \theta_{zz}+2\rho_z\theta_z=-\<\hat{E},f\>.\]
In particular, $\theta_{zz}+2\rho_z\theta_z$ is real.  Therefore $\theta_{zz}+2\rho_z\theta_z\equiv0$, and we see that $\<E,f\>\equiv0$  by \eqref{eq-Ef}. Hence \[\<E,f_z\>=\<E,f_{zz}\>=\<E,\Omega\>=0.\] So $\<\J\Omega_1,\Omega\>=-\<E,\Omega\>=0$.  As a consequence $\sinh\theta\equiv0$ and hence $\theta_z\equiv0$. So we see that $f$ is in the space spanned by
$\{f,f_z,f_{\bar z}, \psi_3\}|_{z=0}$, that is, $f$ is not full, which is a contradiction.

Finally consider item (3).  In this case $f$ is homogeneous by \cite{Bryant2,Kenmotsu}. Thus $\rho$, $|\Omega_1|$ and $|\Omega_2|$ are all constant. Since $f$ is full in $S^n$  with $n>3$, we have $|\Omega_2|\neq0$; otherwise, we will have $\Omega_2\equiv0$ and hence $\nabla^{\perp}_z\Omega_1\equiv0$ by the Codazzi equation. Then $f$ is contained in some $S^3=S^n\cap  {\Span}\{f,\text{Re} f_z, \text{Im} f_{z}, \Omega_1\}|_{z=0}$, a contradiction.

By \eqref{eq-integrability} and \eqref{eq-Omega-L}, we obtain
\[\mathcal{L}(\Omega_1)=2(1+4e^{-4\rho})\Omega_1,\ \mathcal{L}(\Omega_2)=2(1-4e^{-4\rho})\Omega_2.\]
So $\Omega_1$ is an eigensection belonging to the eigenvalue $-2(1+4e^{-4\rho})<-2$ and $\Omega_2$ is an eigensection belonging to the eigenvalue $-2(1-4e^{-4\rho})$.
Since
 {\[|\Omega|^2=|\Omega_1|^2+|\Omega_2|^2>|\Omega_1|^2-|\Omega_2|^2=1,\]} by the Gauss equation \eqref{eq-integrability} we obtain   $0=K=1-4e^{-4\rho}|\Omega|^2<1-4e^{-4\rho}$. So we have $-2(1-4e^{-4\rho})\in(-2,0)$.
Consequently $\Ind(f)\geq (n+1)+2$.
\end{proof}

From the above proof we see that if $f$ is homogeneous, then either it is contained in some great $S^3\subset S^n$, or its Jacobi operator has an eigenvalue in the interval $(-2,0)$.
Recall\footnote{See \cite{Weiner} for $S^3$ or  {\cite[Theorem 3.17]{NS2}} for $S^n$. Note that in \cite{NS2} the definition of Willmore stability is stronger than the one used here. In \cite{KW1} we used an idea from \cite{Kusner} to give a proof simpler than the one in \cite{NS2}.} 
that for a minimal surface in $S^n$, it is Willmore unstable if its Jacobi operator has an eigenvalue in the interval $(-2,0)$.
Since the Clifford torus is the only homogeneous minimal torus in $S^3$ (see \cite{CDK,Lawson} for a proof and \cite{Brendle,Brendle2,CDK,Lawson,Urbano} for related results), and since it is Willmore stable in  $S^n$  {(see \cite{NS2,Weiner}),} we obtain:
\begin{corollary}
Let  $f:T^2\rightarrow S^n$ be a homogeneous minimal torus with $\mathcal{H}\neq0$. Then it is Willmore stable if and only if it is congruent to the Clifford torus in some great $S^3\subset S^n$.
\end{corollary}

\subsection{The superminimal case of vanishing Hopf differential $\mathcal{H}$}
We have two subcases: the general case, and the case in $S^4$.
\begin{proposition} \label{prop-H-va-sn}Let   $f:T^2\rightarrow S^n$ for $n>4$ be a  full minimal torus with $\mathcal{H}\equiv0$. Then $\Ind(f)\geq n+3$.
\end{proposition}

\begin{proof} From \eqref{eq-Omega-L} we have
\[\mathcal{L}(\Omega_1)=2\Omega_1,\ \mathcal{L}(\Omega_2)=2\Omega_2.\]
So it suffices to prove that the dimension of
$\Span_{\mathbb R}\{\Omega_1,\Omega_2,e_0^\perp,\cdots, e_n^\perp\}$
is $n+3$. Assume that there exist some real constants $\hat b_1,\hat b_2, c_0,\cdots,c_n$ such that $E=\sum_{j=0}^{n}c_je_j$ and
\[\hat b_1\Omega_1+\hat b_2\Omega_2+E^\perp=c\Omega+\bar c\bar\Omega+E^\perp=0, \hbox{ with } c=\frac{1}{2}(\hat b_1-i\hat b_2).\]
 Taking derivatives and using \eqref{eq-ENz}, we obtain
\[\begin{split}0& =c\nabla^{\perp}_z\Omega-2\bar ce^{-2\rho}\<\bar\Omega,\Omega\>f_{\bar z} -2e^{-2\rho}\<E,\Omega\>f_{\bar z}-2e^{-2\rho}\<E,f_{\bar z}\>\Omega.\\
\end{split}\]
Therefore
\begin{equation}\label{eq-E-Omega}
c\nabla^{\perp}_z\Omega-2e^{-2\rho}\<E,f_{\bar z}\>\Omega=0,\ \quad  \bar c \<\bar\Omega,\Omega\>+ \<E,\Omega\>=0.\end{equation}
So either $c=0$, or $\nabla^{\perp}_z\Omega\parallel \Omega$. In the latter case, it is easy to see  by the structure equation  \eqref{eq-moving} that $f$ is now contained in some great $S^4$ contained in   {\[S^n\cap Span\{f,\text{Re} f_z, \text{Im} f_{z}, \text{Re} \Omega, \text{Im}\Omega\}|_{z=0}.\]} Hence $c=0$ and  $E=0$.
\end{proof}

The $S^4$ case is a bit more subtle. It is easy to show that there exists no $E\in \mathbb C^5$ such that $E^\perp=\bar \Omega$.  While it seems impossible to get a better estimate of the index by the above method, the work of Ejiri \cite{Ejiri1983} can be modified to show  {the following proposition.}

\begin{proposition} \label{prop-H-va-s4} Let   $f:T^2\rightarrow S^4$  be a  minimal torus with $\mathcal{H}\equiv0$. Then $\Ind(f)\geq 12$.
\end{proposition}
\noindent

We defer its proof to Section 4.1, where we also consider the case of higher genus surfaces by adapting Ejiri's work \cite{Ejiri1983} on 2-spheres; there we also consider sharper index estimates one gets for minimal tori with $\mathcal{H}\equiv0$ and for higher genus surfaces, using the method of \cite{MU}.

Note that El Soufi and Ilias \cite{ElSoufi2} show for any minimal torus $f$ with $\mathcal{H}\equiv0$ in $S^4$ that its first Laplacian eigenvalue satisfies $\lambda_1<2$; by \cite{KW1} (or by  {\cite[Proposition 3.13]{Karp})}, if $f$ is composed with the embedding $S^4\subset S^n$ as a great sphere, its index $\Ind(f,n)$ in $S^n$ satisfies
\[\Ind(f,n)\geq \Ind(f,4)+2(n-4)\geq 2n+4>n+3.\]
Thus we conclude   {the following theorem.}
\begin{theorem} \label{thm-H-va}Let $f:T^2\rightarrow S^n$ for $n\ge 4$ be a minimal torus with $\mathcal{H}\equiv0$. Then \[\Ind(f,n)\geq n+3.\]
\end{theorem}

\begin{remark} \
\begin{enumerate}\item
The equilateral minimal torus in $S^5\subset S^n$ studied by Bryant, Itoh, Montiel and Ros \cite{Bryant3, Itoh, Montiel, KW1}, has vanishing Hopf differential and index $n+3$ \cite{KW1}.  For $n\ge5$, it is the only minimal torus besides the Clifford torus immersed in $S^n$ by first eigenfunctions of the Laplacian \cite{ElSoufi2}.
It is natural to conjecture that it is the only minimal torus --- and perhaps the only orientable minimal surface of genus $g\ge1$ --- in $S^n$ with vanishing Hopf differential and index $n+3$.
\item From \cite{Ejiri1983} it is easily seen that $\Ind(f)>n+3$ for a  {non-totally umbilic} minimal $2$-sphere $f$ in $S^n$, which necessarily \cite{Calabi} has $\mathcal{H}=0$.
\item By \cite{Barbosa,Bryant,Calabi}, minimal tori in $S^4$ with $\mathcal{H}=0$ have area --- and hence Willmore energy --- greater than $8\pi$, so are {\em not} candidates for the Willmore minimizer among tori in $S^4$.
We point out \cite{ElSoufi1, Li-y, Marques, Marques2, Ros} for further discussion of the Willmore conjecture, and we refer to \cite{Mon} for more details on Willmore 2-spheres in $S^4$ with
vanishing $\mathcal{H}$.
\end{enumerate}
\end{remark}

{\section{Remarks on the higher genus case and the nonorientable case}}
In this section we estimate the index of higher genus minimal surfaces in $S^n$, depending on whether or not the Hopf differential $\mathcal{H}$ vanishes.
\subsection{The case with vanishing Hopf differential $\mathcal{H}$}
If we focus on surfaces in $S^4$ for this case,
the arguments of Ejiri in \cite{Ejiri1983} can be modified to give a lower bound for the index.
Here $M$ denotes a closed Riemann surface of genus $g\geq 0$.

Let   $f:M\rightarrow S^4$  be a minimal surface with $\mathcal{H}\equiv0$. The normal bundle of $f$ has a natural almost complex structure $\J$ as shown in Section 3 (we use
$J$ to denote the complex structure on $TM$ for the conformal structure induced by $f$).  Define a Cauchy-Riemann operator $\mathfrak{D}$ on smooth sections of the normal bundle whose kernel consists of the $\J$-holomorphic sections:
\[\mathfrak{D}_XV:=\nabla^{\perp}_{JX}V-\J\nabla^{\perp}_X V,\  {\hbox{ for all }} V\in\Gamma(\NM),\ X\in\Gamma(TM).\]
Then we have the following
\begin{lemma} {\rm ( {\cite[Lemma 3.2]{Ejiri1983}})} In terms of this Cauchy-Riemann operator, the second variation quadratic form on a minimal $f:M\rightarrow S^4$
with $\mathcal{H}\equiv0$ is
\begin{equation}\label{eq-Ejiri}
Q(V,V)=-2\int_M|V|^2\dd A+\frac{1}{2}\int_M|\mathfrak DV|^2\dd A.
\end{equation}
\end{lemma}
This implies immediately  {the following result.}
\begin{theorem} \label{thm-Ejiri} Let   $f:M\rightarrow S^4$  be a minimal immersion of a closed Riemann surface $M$ with $\mathcal{H}\equiv0$. Then the first eigenvalue $\alpha_1$ of the area-Jacobi operator $\mathcal{L}$ satisfies
\begin{equation}\label{eq-Ejiri-2}
  \alpha_1=-2  \,\,\,\,\,\,\, \hbox{with multiplicity} \,\,\,\,\,\,\, m_1=\hbox{$\dim$}_{\mathbb R}\{\hbox{$\J$-holomorphic sections of $\NM$}\}. 
\end{equation}
\end{theorem}

\begin{theorem} \label{thm-g}  Let   $f:M\rightarrow S^4$  be a minimal immersion of a closed Riemann surface $M$ of genus $g\geq 0$ with $\mathcal{H}\equiv0$. If $f$ is not totally geodesic, then $\Ind(f)\geq 2g+10$.
\end{theorem}

\begin{proof} By Theorem \ref{thm-Ejiri}, we need to estimate $m_1$, adapting Ejiri's argument from the genus $0$ case \cite{Ejiri1983}. We deduce (following  {\cite[pages 131-132]{Ejiri1983}})
\[m_1=2\left(\chi(M,\NM)+\dim H^1(M, \NM)\right),\]
where $\chi(M,\NM):=-\dim H^1(M,\NM)+\dim H^0(M,\NM)$, and $H^j(M,\NM)$ denotes
the $j$th cohomology group with coefficients in the sheaf $\mathcal{O}(\NM)$ of germs of local $\J$-holomorphic sections of $\NM$.

Let $\widetilde{c}_1$ be the first Chern class of $\NM$ and $c_1$ the first Chern class of $TM$. By the Riemann-Roch theorem, we have
\[\chi(M,\NM)=\widetilde{c}_1([M])+\frac{1}{2}c_1([M])=\frac{1}{2\pi}\int_M(1-K)\dd A+\frac{1}{2}\chi(M)=\frac{A(M)}{2\pi}-\frac{1}{2}\chi(M),\]
where $A(M)$ is the area of $f$. Since $f$ has $\mathcal{H}\equiv0$, it comes via {\em twistor projection} \cite{Calabi, Bryant} into $S^4$ from some horizontal
(anti-)holomorphic curve $\mathcal F$ in $\mathbb{C}P^3$. Thus the area of $f$ (see \cite{Barbosa}) is equal to $4\pi \deg(\mathcal F)\geq 12\pi$ if $f$ is not totally geodesic, and hence we have
\[\Ind(f)\geq m_1\geq 12-\chi(M)=2g+10.\]
\end{proof}
Setting $g=1$ in Theorem \ref{thm-g} gives the proof of Proposition \ref{prop-H-va-s4}.\vspace{2mm}
\begin{remark} Ejiri's estimate \cite{Ejiri1983} for the index of minimal $2$-spheres can be generalized to all minimal surfaces in $S^{2n}$ obtained by twistor projection to $S^{2n}$ of horizontal (anti-)holomorphic curves in the twistor bundle $\mathfrak{J}S^{2n}$ of $S^{2n}$ --- an idea  stemming from Calabi \cite{Calabi} --- but to obtain a better index estimate in high codimension, one would also need to get more precise estimates for the areas of such minimal surfaces when $ g\geq 1$, which lies beyond the scope of this paper.
\end{remark}
We next observe that a better estimate for the index in the higher genus case can be obtained by combining   { \cite[Theorem 4]{MU} and \cite[Lemma 6]{MU}} of Montiel and Urbano.
\begin{theorem} \cite{MU} Let   $f:M\rightarrow S^4$  be a full minimal immersion of a closed Riemann surface $M$ of genus $g\geq 0$ with $\mathcal{H}\equiv0$. Then
\[\Ind(f)=2g-2+\frac{ {\hbox{Area}}(f)}{\pi}\geq 2g+10+4\left[\frac{3g+3}{4}\right]\geq 5g+10.\]\end{theorem}

 \subsection{The case with nonvanishing Hopf differential  $\mathcal{H}$}

In this case,  $\Omega$ is no longer a global section of $\NM\otimes\underline{\mathbb C}$ when $g(M)\geq 2$, nor do we  have $\<\Omega,\Omega\>\equiv1$. So the treatment in Section 3 fails. But in this case, it seems likely that the first eigenvalue $\alpha_1$ of $\mathcal{L}$ still satisfies $\alpha_1<-2$, which would imply that $\Ind(M)\geq 1+n+1=n+2$.
  It therefore naturally leads us to pose the  {following conjecture.}
 \begin{conjecture}
 	Let $f:M\rightarrow S^n$ be an orientable,  {non-totally geodesic}, closed minimal surface. Then $\Ind(M)\geq n+2$,
 	with equality holding if and only if $M=T^2$ and $f$ is congruent to the Clifford torus.
 \end{conjecture}
 \begin{remark}
 	The conjecture holds if $M$ has genus $0$ (by \cite{Ejiri1983}) or genus $1$ (except the equality case when $n>4$ and $\mathcal H\neq0$) by Theorem \ref{prop-H-no} and Theorem \ref{thm-H-va}.
 \end{remark}

\subsection{On nonorientable minimal surfaces}

There has been recent progress estimating the index for immersed nonorientable surfaces.
In \cite{Karp}, the index of minimal real projective planes in $S^n$ is computed. In particular, the Veronese ${\mathbb R}P^2$ in $S^4$ has index $5$.
Then we know by \cite{KW1} or  {\cite[Proposition 3.13]{Karp}} that it has index $n+1$ for all $n\geq4$ if we embed  $S^4$  {totally geodesically} into $S^n$  for $n>4$, 
since it is immersed by its first eigenfunctions (see also \cite{Karp2}).
In particular, among minimal real projective planes in $S^n$ the following  {index characterization} of Veronese was recently offered by Karpukhin:
\begin{theorem} \cite{Karp}
Let $\Sigma$ be a minimal ${\mathbb R}P^2$ in $S^n$. Then   {$ \Ind(\Sigma)\geq n+1,$}
with equality holding  {if and only if} $\Sigma$ is congruent to the Veronese ${\mathbb R}P^2$ in $S^4$.

Moreover, if $\Sigma$ is not congruent to the Veronese ${\mathbb R}P^2$ in $S^4\subseteq S^n$, we have the following
 \begin{enumerate}
 \item $\Ind(\Sigma)\geq 2n+1$, if $\Sigma$ is fully contained in some $S^4\subset S^n$;
 \item  $\Ind(\Sigma)\geq m(m+2)-3+2(n-2m)=2n+(m+1)(m-3)$, if it is fully contained in $S^{2m}\subset S^n, m\geq3$.
\end{enumerate}
In particular, $\Ind(\Sigma)\geq 2n>n+3$ for any $n\geq 4.$
\end{theorem}

\begin{proof}By  {\cite[Proposition 1.11]{Karp}},
\[\Ind(\Sigma)=\frac{1}{2}\Ind(\tilde \Sigma),\]
where $\tilde\Sigma$ is the orientation double covering of $\Sigma$. By the main Theorem of \cite{Ejiri1983}, \[\Ind(\tilde\Sigma)\geq 2m(m+2)-6\] if $\Sigma$ is fully immersed in $S^{2m}$.

If $m\geq3$, we see that the first eigenvalue of $\Sigma$ is  {smaller than } $2$. So (by   {\cite[Proposition 3.13]{Karp}}, or by \cite{KW1})  \[\Ind(\Sigma)\geq m(m+2)-3+2(n-2m).\]

If $m=2$,  by \cite{MU}, $\Ind(\tilde\Sigma) =4d-2$ in $S^4$, with  $d$ being the degree of the twistor lift of $\tilde\Sigma$. By \cite{Ejiri-E}, $d\geq3$ and $d=3$
if and only if $\Sigma$ is congruent to the Veronese ${\mathbb R}P^2$ in $S^4$.
Then we have (again by  {\cite[Proposition 3.13]{Karp}}  or \cite{KW1})
\[\Ind(\Sigma)\geq 2d-1+n-4\geq n+1\] with equality holding  {if and only if} $d=3$.
Moreover, we have $d$ is odd by \cite{Ga}.  So when $\Sigma$ is not congruent to the Veronese ${\mathbb R}P^2$ in $S^4$, $d\geq 5$ and hence \[\Ind(\Sigma)\geq 10-1+2(n-4)=2n+1.\]
\end{proof}
\begin{remark} Our methods have recently been adapted to the nonorientable case:
\begin{enumerate}
\item
Karpukhin   {et al.} \cite{Karp2}
noted that for particular minimal Klein bottles in $S^4$ the real part of the Calabi differential
$\CC=\Omega\dd z^2=f_{zz}^\perp\dd z^2$ defines a global section of
the (real) normal bundle. They use this to show that the Lawson bipolar minimal Klein bottle $\tilde\tau_{3,1}$ in $S^4$ has index $6$.
Moreover, since $\tilde\tau_{3,1}$  has first Laplacian eigenvalue $\lambda_1=2$, it has index $n+2$ in $S^n$ for all $n\geq4$ if we embed $S^4$  {totally geodesically} into $S^n$ for $n>4$ (see  {\cite[Proposition 3.11]{Karp2}}).

\item It is interesting to compare the above  with  results for orientable minimal surfaces:
    \begin{enumerate}
 \item the  totally geodesic round $2$-sphere in $S^n$ has index $n-2$;
 \item the Clifford torus in $S^3\subset S^n$ has index $n+2$ in $S^n$;
 \item the equilateral minimal torus  in $S^5\subset S^n$ has index $n+3$ in $S^n$ ( {see \cite{KW1}});
 \item the Lawson surface $\xi_{g,1}$  in $S^3\subset S^n$ has index $2g+n$ in $S^n$ ( {see \cite{KW1} and also \cite{choesoret,KaW,Karp}}).
\end{enumerate}
\item
Even more recently, Medvedev \cite{Medvedev} applied our methods to the case of nonorientable {\em free boundary minimal surfaces} in the unit 4-ball ${\mathbb B}^4$,
showing that the critical M\"obius band $M\subset {\mathbb B}^4$ has index $5$, and conjecturing that this characterizes $M$.
\end{enumerate}
\end{remark}
{\footnotesize
\def\refname{References}

}
\end{document}